\documentclass[12pt]{amsart}
\usepackage{latexsym, amsmath, amssymb,amsthm,amsopn,amsfonts}
\usepackage{version}
\usepackage{epsfig,graphics,color,graphicx,graphpap}
\usepackage{amssymb}
\usepackage{todonotes}
\usepackage{listings}

\usepackage[framemethod=tikz]{mdframed}

\usepackage[citecolor=blue]{hyperref}
\usepackage[framemethod=tikz]{mdframed}

\numberwithin{equation}{section}
\newtheorem{theorem}{Theorem}[section]

\newtheorem{proposition}[theorem]{Proposition}
\newtheorem{lemma}[theorem]{Lemma}

\newtheorem{remark}{Remark}[section]

\newcommand{\LN}{\left\|}
\newcommand{\RN}{\right\|}
\newcommand{\LV}{\left|}
\newcommand{\RV}{\right|}
\newcommand{\LC}{\left(}
\newcommand{\RC}{\right)}

\newcommand{\p}{\partial}

\DeclareMathOperator{\supp}{supp}
\DeclareMathOperator{\diam}{diam}

\newcommand{\R}{\mathbb R}
\newcommand{\Kn}{\mathsf{Kn}}
\newcommand{\rd}{\mathrm{d}}

\title[ ]{Inverse problems for the stationary transport equation in the diffusion scaling}
\author[Lai]{Ru-Yu Lai}
\address{School of Mathematics, University of Minnesota, Minneapolis, MN 55455, USA}
\curraddr{}
\email{rylai@umn.edu }

\author[Li]{Qin Li}
\address{Department of Mathematics, University of Wisconsin-Madison, Madison, WI 53705, USA}
\curraddr{}
\email{qinli@math.wisc.edu}
\thanks{}

\author[Uhlmann]
{Gunther Uhlmann}
\address{Department of Mathematics, University of Washington, Seattle, WA 98195, USA;}\address{HKUST Jockey Club Institute for Advanced Study, HKUST, Clear Water Bay, Kowloon, Hong Kong.}
\email{gunther@math.washington.edu}
\thanks{  }

\date{\today}

\begin{document}
\maketitle
\tableofcontents
\begin{abstract}
	
	We consider the inverse problem of reconstructing the optical parameters of the radiative transfer equation (RTE) from boundary measurements in the diffusion limit.
	In the diffusive regime (the Knudsen number $\Kn\ll 1$), the forward problem for the stationary RTE is well approximated by an elliptic equation. However, the connection between the inverse problem for the RTE and the inverse problem for the elliptic equation has not been fully developed. This problem is particularly interesting because the former one is mildly ill-posed , with a Lipschitz type stability estimate, while the latter is well known to be severely ill-posed with a logarithmic type stability estimate.
	In this paper, we derive stability estimates for the inverse problem for RTE and examine its dependence on $\Kn$. We show that the stability is Lipschitz in all regimes, but the coefficient deteriorates as $e^{\frac{1}{\Kn}}$, making the inverse problem of RTE severely ill-posed when $\Kn$ is small. In this way we connect the two inverse problems.
Numerical results agree with the analysis of worsening stability as the Knudsen number gets smaller.

\end{abstract}

\section{Introduction}

Optical tomography (OT) is a technique that uses low-energy visible or near-infrared light in the wavelength region ($650$nm $\sim$ $900$nm) to illuminate highly scattering media \cite{Arridge1999}. In OT, based on measurements of scattered and transmitted light intensities on the surface of the medium, a reconstruction of the spatial distribution of the optical properties, for instance, absorption coefficient, $\sigma_a$, and scattering coefficient, $\sigma_s$, inside the medium is attempted. OT has potential applications into a variety of  science and engineering fields, including oceanography, atmospheric science, astronomy, and neutron physics \cite{McCor}. More recently the application of OT to medical imaging has received considerable attraction, where visible or near-infrared light are sent into tissues, and the optical parameters are reconstructed to  distinguish between healthy and unhealthy tissues.


However, the problem has not been fully understood mathematically. In fact, there are a variety of forward models for describing photon propagation. The two widely used models are the radiative transfer equation (RTE, also known as the linear Boltzmann equation) and the diffusion equation (DE). What is intriguing here is that the two models are good approximation to each other in the diffusion regime where the Knudsen number ($\Kn$) in RTE is small, but the corresponding inverse problems are proved to be mildly ill-posed and severely ill-posed, respectively. The aim of this article is to study the connection between the two models in the inverse problem setting. More precisely we study the stability of the parameter reconstruction for RTE, and in particular its dependence on $\Kn$. We will show  that despite the stability is Lipschitz like, the Lipschitz constant blows up in an exponential fashion  for small  $\Kn$, showing the severe ill-posedness when RTE is in the DE regime.

We now give a brief review of both models.

\subsection{RTE and its inverse problem}
We now write down the time-independent radiative transfer equation (RTE), which models photon transport in tissues. For survey papers on the subject see ~\cite{Arridge1999, Bal_transport, Somersalo2005}: for $x\in\Omega$ and $v\in V$,
\begin{align}\label{RTE}
\begin{cases}
      v\cdot \nabla_x f(x,v) + (\sigma_a(x)+\sigma_s(x)) f(x,v) - \sigma_s(x) \int_V p( v',v) f(x,v')dv'=0,\\
      f|_{\Gamma_-} = \text{given data}.
      \end{cases}\,
\end{align}
In equation \eqref{RTE}, $f(x,v)$ is defined on the phase space and is the density of particles at position $x$ with velocity $v$ in an open set $V$,  $\sigma_a$ and $\sigma_s$ are two optical parameters representing the absorption coefficient and the scattering coefficient. They model how likely a photon particle is absorbed or scattered by the media. The scattering phase function is
\[
k(x,v,v')=\sigma_s(x)p(v,v')
\]
that defines the probability that during a scattering event, a photon from direction $v'$ is scattered in the direction $v$ at the point $x$. We also define the total absorption coefficient $\sigma=\sigma_a+\sigma_s$. It measures how likely a photon with velocity $v$ disappears (from either being absorbed or scattered).

The physical domain, denoted by $\Omega$, is a subset of $\mathbb{R}^n,\ n=2,3$. It is a bounded open convex set with $C^1$ boundary $\p\Omega$. Let $V$ be the velocity domain which is an open set in $\R^n$. We assume that there are constants $M_1$ and $M_2$ such that
\begin{equation}\label{eqn:v_bound}
0<M_1 < |v|< M_2
\end{equation}
for all $v$. In addition, $\Gamma_+$ and $\Gamma_-$ are used to denote the coordinates on the physical boundary associated with the outgoing and incoming velocities, respectively:
$$\Gamma_{\pm}=\{(x,v)\in \p \Omega\times V: \pm \ n_x\cdot v > 0\},$$
where $n_x$ is the unit outer normal to $\p\Omega$ at the point $x\in \p\Omega$.  The boundary condition, therefore, is placed on $\Gamma_-$.

The inverse problem for the transport equation amounts to reconstructing the unknown optical parameters $\sigma_a+\sigma_s$ and $k$ from the boundary measurements. Namely, we call the map, from the incoming density $f|_{\Gamma_-}$ to the outgoing density $f|_{\Gamma_+}$, the \textit{albedo} operator:
$$
     \mathcal{A}:f|_{\Gamma_-}\rightarrow f|_{\Gamma_+}; 
$$
and reconstruct $\sigma_a+\sigma_s$ and $k$ from the entire albedo map.

The theoretical approach on the reconstruction of the optical parameters $(\sigma, k)$ is based on the singular decomposition of the Schwartz kernel $\alpha=\mathcal{A}_1+\mathcal{A}_2+\mathcal{A}_3$ of the albedo operators $\mathcal{A}$ where $\mathcal{A}_j$, $j=1,2,3$, are described in section 2.2. For detailed discussion, we refer the readers to \cite{CS2,  CS98, Stefanov_2003}. Here we briefly discuss the approach to reconstruct the coefficients $\sigma$ and $k$ through the study of the kernel. 

For the time-independent problem in dimension $n\geq 3$, one can observe that $\mathcal{A}_1$ and $\mathcal{A}_2$ are delta functions, and $\mathcal{A}_3$ is a locally $L^1$ function. Thus, $\mathcal{A}_3$ can be distinguished from $\mathcal{A}_1+\mathcal{A}_2$. 
Moreover, since $\mathcal{A}_1$ and $\mathcal{A}_2$ have different degrees of singularities, $\sigma$ can be reconstructed from $\mathcal{A}_1$ and then based on the knowledge of $\sigma$, one can further recover $k$ from $\mathcal{A}_2$. As for the time-dependent problem, a similar procedure can be used to recover both $\sigma$ and $k$ for any dimension $n\geq 2$ without additional assumption on $k$. However, in the two-dimensional case for the time-independent problem, since $\mathcal{A}_2$ is a locally $L^1$ function as well, it can not be distinguished from $\mathcal{A}_3$. Therefore, the same approach does not work for recovering $k$. Although, $\mathcal{A}_1$ and $\mathcal{A}_2$ are still distinguishable, as a result, one can still recover $\sigma$ for the stationary case in two dimensions.



The inverse stationary transport problem has less data compared to the time-dependent problem due to the absence of the time variable. 
This stationary inverse problem is overdetermined in dimension $n\geq 3$ since the kernel of $\mathcal{A}$ depends on higher dimensions than what the parameters $\sigma$ and $k$ depend on.
The unique result for both optical parameters $\sigma$ and $k$ was studied in \cite{CS3, CS98}. The associated stability estimate was derived in \cite{Bal14, Bal10, Bal18, Wang1999} in dimension $n\geq 3$. 
However, in $n=2$, it is only formally determined for the recovery of $k$ and is still overdetermined for $\sigma$. It was showed in 
\cite{SU2d} that, under a smallness assumption on $k$, the coefficients $\sigma$ and $k$ can be uniquely determined.
Finally, we note that the inverse boundary value problem for time-dependent transport equation, unique determination of $\sigma$ and $k$ were studied in \cite{CS1, CS2} for any dimension $n\geq 2$.


\subsection{Diffusion equation and its inverse problem}

Depending on the relation between the scattering coefficient $\sigma_s$ and the absorption coefficient $\sigma_a$, the RTE sometimes can be approximated by the diffusion approximation. More specifically, in scatter dominated materials the diffusion approximation holds valid, but in materials where $\sigma_a$ dominates or is comparable to $\sigma_s$, the diffusion approximation is not a suitable model. The first case is seen in breast tissue where $\sigma_s$ is much larger than $\sigma_a$ at appropriate wavelengths ($650$nm $\sim$ $900$nm)~\cite{DDA}. More works on breast images studies can be found in~\cite{Colak, Fantini}. However, when the absorption coefficient of the medium is similar to the scattering coefficient, the diffusion approximation might not be a good approximation to describe the photo migration in biological tissues.
For example, in the blood vessels or organs with a high blood perfusion, such as in the liver, the approximation does not hold at any wavelength. 
We refer the interested readers to \cite{HA1996}, for instance.


Assume that scattering effect dominates, then the diffusion equation is modeled by:
\begin{align}
-C\nabla_x\cdot ( \sigma_s^{-1}  \nabla_x \Phi(x)) +\sigma_a \Phi(x) =0\, 
\end{align}
with a constant $C$. Note that the diffusive media takes the reciprocal of $\sigma_s$, and the photon intensity is defined by:
$$
\Phi(x)=\int_V f(x,v)\rd{v}\,.
$$
We refer to \cite{Bal_transport} for a detailed discussion on the transport equation in the diffusive regime, and only cite the results here.

Now we consider the strong scattering case and we define
\begin{align}\label{sigma_a}
 \sigma_{\Kn}(x)= \Kn\sigma_a+\Kn^{-1}\sigma_s\ \ \ \ \hbox{for }0<\Kn\ll 1 .
\end{align}
and $k=\Kn^{-1}\sigma_s(x)$ with $p(v',v)$ set as $1$. Hereafter, we replace $\sigma$ by $\sigma_{\Kn}$ and $k$ by $\Kn^{-1}\sigma_s$ in \eqref{RTE} with time-independent, then we have the following equations:
\begin{align}\label{diffusion}
    \left\{ \begin{array}{lc}
    v\cdot \nabla_x f(x,v) +\sigma_{\Kn}(x)f(x,v) - \Kn^{-1} \sigma_s(x)\int_V  f(x,v')dv'=0 & \hbox{in }\Omega
    \times V, \\
    f|_{\Gamma_-} =f_- . & \\
    \end{array}\right. 
\end{align}

The following result can be proved from~\cite{BardosSantosSentis:84,BLP:79,LiLuSun2015JCP,LLS_geometry,WG2014}:
\begin{theorem}\label{thm:diffuse}
Suppose that $f$ solves~\eqref{diffusion}. As $\Kn \to 0$, $f (x, v)$ converges to $\rho(x)$, where $\rho(x)$ solves the diffusion equation:
\begin{equation}\label{eqn:elliptic}
 C\nabla_x\cdot (\sigma_s^{-1} \nabla_x\rho) -\sigma_a\rho =0\,.
\end{equation}
Here C is a constant depending on the dimension of the problem. The boundary condition is determined by:
\begin{equation}
 \rho|_{\p\Omega} = \xi_f\,
 \end{equation}
with $\xi_f (x_0) = f^l_{z\to\infty}$ where $f^l$ solves the boundary layer equation:
 \begin{equation*}
v\partial_zf^l = \sigma_s\left(\int f^l(z,v')\rd{v'} - f^l\right)\,,\quad z \in [0,\infty)\quad \text{with}\quad f^l|_{z=0}  = f_-(x_0,v).
 \end{equation*}
 Moreover, one has
\begin{enumerate}
 \item $\int (v\cdot n_{x_0})f(x_0,v)\rd{v} = \Kn \sigma_s^{-1}\p_{n_{x_0}}\rho(x_0) +\mathcal{O}(\Kn^2)$; and
 \item  if $f_-(x,v) = f_-(x)$ is independent of $v$ for all $x\in\p\Omega$, then $ \rho|_{\p\Omega} = \xi_f = f_-$.
\end{enumerate}
\end{theorem}

\begin{remark}
Define the averaged albedo operator for~\eqref{diffusion} by
\begin{equation*}
\mathcal{A}[f_-] =  \Kn^{-1}\int (v\cdot n_x)f(x,v)\rd{v}\,,
\end{equation*}
and the Dirichlet-to-Neumann (DtN) map for the diffusion equation~\eqref{eqn:elliptic} as:
\begin{equation*}
\Lambda[ \rho|_{\p\Omega}] = \p_{n_{x}}\rho(x)\,,
\end{equation*}
assuming that the incoming boundary condition $f_-$ is homogeneous in $v$, then according to Theorem~\ref{thm:diffuse}, $f_-(x_0) = \rho(x_0)$ on $\p\Omega$, and the averaged albedo operator converges to the DtN map, meaning:
\begin{equation}
\mathcal{A}[f_-] -\Lambda[\rho|_{\partial\Omega}]=  \Kn^{-1}\int (v\cdot n_x)f(x,v)\rd{v} - \p_{n_{x_0}}\rho(x_0) = \mathcal{O}(\Kn)\,.
\end{equation}
\end{remark}

It is well-known that the inverse problem for the elliptic equation~\eqref{eqn:elliptic}, that is, using the DtN map to recover the coefficients in \eqref{eqn:elliptic}, is severely ill-posed. In particular, it has a  logarithmic type of stability estimate. This kind of estimate was first derived by Alessandrini in \cite{A1}  
and was shown to be optimal in \cite{Mandache}. For reviews of the stability issue and the Calder\'on's problem, we refer to \cite{Astability, uhlmann2009electrical}. 
In contrast to the inverse problem for the elliptic equation, the inverse problem of recovering the media in RTE \eqref{diffusion} has a Lipschitz type stability, see \cite{Bal14, Bal10, Bal18, Wang1999}.




\subsection{Main result}\label{mainresult}

In this article, we are interested in bridging the stability estimates for these two inverse problems. 
We are motivated by the study of increasing stability behavior for several elliptic inverse problems when the frequency gets higher. It is known that the logarithmic stability makes reconstruction algorithms challenging since a small error in the data could be magnified exponentially in the numerical reconstruction.
The research on increasing stability therefore arises from the desire to design a more reliable reconstruction algorithm. 
Its central idea is to obtain stability estimates that contain two parts: one is Lipschitz, the other is logarithmic, and to derive associated coefficients that explicitly depend on certain parameters in the forward model. With the parameters chosen in a suitable range, one part of the estimate dominates the other, which leads to increasing stability or stability deterioration.  
The problems along the line of thinking in different contexts have been addressed in \cite{I, ILW16, INUW, iw14, Ldiff, Liang, NUW}, and in~\cite{time_harmonic, BalMonard_time_harmonic} the authors particularly studied the stability of the inversion with respect to the modulation frequency in time-harmonic setting for RTE, and found that the increasing of the frequency brings more details in the recovery. See also \cite{Ibook} for the study of increasing stability in different problems.
 
In this paper we determine the stability of inverting RTE by using the albedo operator, we trace its explicit dependence on $\Kn$. 
Assume that the media $(\sigma_{\Kn}, \Kn^{-1}\sigma_s)$ is admissible (to be clear in Section \ref{wellposedness}) to have the boundary value problem~\eqref{diffusion} well-posed:
\begin{align}\label{wellpose}
    \|\tau \sigma_{\Kn}\|_{L^\infty}<\infty,\ \   
    \left\|\tau  \int_V \Kn^{-1} \sigma_s(x) dv\right\|_{L^\infty}<\infty,\ \ \hbox{and}\ \ \sigma_{\Kn}\geq \int_V \Kn^{-1}\sigma_s(x) dv 
\end{align}
for a.e. $x\in \Omega$, and let
\begin{equation}\label{eqn:P_space}
\mathcal{P}=\{u\in H^{3/2+r'}(\Omega):\ u\geq 0,\ \supp(u)\subset\Omega,\ \|u\|_{H^{3/2+r'}(\R^n)}\leq M_3\}
\end{equation}
for some $r'>0 $. Here $\text{supp}(u)$ denotes the compact support of a function $u$. 

Our main result is stated as follows and its proof will be given in section \ref{stability_est}.
\begin{theorem}[Stability estimate with explicit $\Kn$ dependence]\label{thm1.1}
Let $\Omega$ in $\mathbb{R}^n,\ n=2,3$ be a bounded open convex set with $C^1$ boundary $\p\Omega$, and let $0<\Kn<1$. Suppose the assumption \eqref{wellpose} holds and that $\sigma_a,\ \sigma_s,\ \tilde{\sigma}_a,$ and $\tilde{\sigma}_s$ are in $\mathcal{P}$. Denote $\mathcal{A}$ and $\tilde{\mathcal{A}}$ are albedo operators associated with media pairs $(\sigma_{\Kn}, \Kn^{-1}\sigma_s)$ and $(\tilde\sigma_{\Kn}, \Kn^{-1}\tilde\sigma_s)$ respectively. Then for some $\theta\in (0,1)$, there exists a constant $C$, independent of $\Kn$, such that the estimate
\begin{align*}
    \|\sigma_{\Kn} - \tilde{\sigma}_{\Kn}\|_{L^\infty} \leq C \Kn^{-1+\theta}  e^{C\theta  \Kn^{-1} } \|\mathcal{A}-\tilde{\mathcal{A}}\|_\ast^\theta 
\end{align*}
holds, where $\|\cdot\|_*$ is the operator norm from $L^1(\Gamma_-,d\xi)$ to $L^1(\Gamma_+,d\xi)$.

Moreover, if
$\Kn< |\log( \|\mathcal{A}-\tilde{\mathcal{A}} \|_\ast)|^{-\alpha} $
for some $\alpha>0$, then 
\begin{align}\label{sigma_s_equ_1_thm}
\|\sigma_s-\tilde{\sigma}_s\|_{L^\infty} 
\leq C\Kn |\log( \|\mathcal{A}-\tilde{\mathcal{A}} \|_*)|^{-\alpha} + C  \Kn^{\theta}  e^{ C\theta{\Kn}^{-1}} \|\mathcal{A}-\tilde{\mathcal{A}} \|_*^\theta\, 
\end{align}
and
\begin{align}\label{sigma_a_equ_1_thm}
\|\sigma_a-\tilde{\sigma}_a\|_{L^\infty} 
\leq C\Kn^{-3} |\log( \|\mathcal{A}-\tilde{\mathcal{A}} \|_*)|^{-\alpha} + C  \Kn^{-2+\theta}  e^{ C\theta{\Kn}^{-1}} \|\mathcal{A}-\tilde{\mathcal{A}} \|_*^\theta\,.
\end{align}
\end{theorem}

This theorem indicates that the stability estimates \eqref{sigma_s_equ_1_thm} and \eqref{sigma_a_equ_1_thm} are exponentially bad when $\Kn$ is small, and within certain range of $\Kn$, the logrithmic illposedness of the Cald\'eron problem is recovered. For more detailed discussion, see remarks in the end of section \ref{stability_est}.


This paper is organized as follows. In section \ref{preliminaries}, we discuss preliminaries and state several known results about the albedo operator decomposition. Section \ref{stability_est} is devoted to the study of stability estimate's dependence on the Knudsen number. Numerical examples are provided in section \ref{numerics} that confirm both the Lipschitz stability and the logarithmic ill-posedness for small $\Kn$, and thus the numerical experiments are in agreement with the statements in Theorem~\ref{thm1.1}.


\section{Preliminaries}\label{preliminaries}
 
In this section, we recall several function spaces and introduce notations, as well as, some known results. They are relevant in our setup and the reconstruction of the optical parameters.

\subsection{Function spaces}
We define the Sobolev spaces $H^s(\R^n)$ in the whole space by 
$$
   H^s(\R^n) = \left\{u\in \mathcal{S}':\ \|\langle D\rangle^s u\|_{L^2(\R^n)} \right\},
$$
where $\langle D\rangle^s u:=\mathcal{F}^{-1}((1+|\xi|^2)^{s/2}\mathcal{F} u )$. Here $\mathcal{F}u$ denotes the Fourier transform of $u$ and $\mathcal{S}'$ is the dual of the Schwartz space $\mathcal{S}$.
In addition, for an open set $U$ in $\R^n$, we define the class of functions in $H^s(\R^n)$ which are restricted in $U$ by
$$
    H^s(U)=\left\{u|_U:\ u\in H^s(\R^n) \right\}.
$$

Moreover, we define the measure on $\Gamma_\pm$ by
\begin{equation}\label{eqn:bdy_measure}
d\xi(x,v)=|n_x\cdot v| d\mu(x)dv
\end{equation}
with the measure $d\mu(x)$ defined on the boundary $\p\Omega$. We denote $L^1(\Gamma_\pm, d\xi)$ to be the space consists of functions $u$ such that $$\int_{\Gamma_\pm}|u(x,v)|d\xi(x,v)<\infty.$$

 
\subsection{Kernel of the albedo operator}\label{wellposedness}   

We consider the boundary value problem with Dirichlet boundary condition for the stationary transport equation:
\begin{align*}
\left\{ \begin{array}{lc}
v\cdot \nabla_x f(x,v) + \sigma(x,v) f(x,v) - \int_V k(x,v',v) f(x,v')dv'=0 & \hbox{in }\Omega
\times V, \\
f|_{\Gamma_-} =f_- . &  \\
\end{array}\right. 
\end{align*}
The pair $(\sigma, k)$ is called \textit{admissible} if 
\begin{align}\label{admissible1}
0\leq \sigma  \in L^\infty(\Omega\times V)
\end{align}
\begin{align}\label{admissible2}
0\leq k(x,v',\cdot)\in L^1(V)  
\end{align}
for almost everywhere (a.e.) $(x,v')\in \Omega\times V$. Moreover,  we define the scattering cross-sections by $\int_Vk(x,v',v)dv$ which is in $L^\infty(\Omega\times V).$ 
The collected data is defined by the albedo operator $$\mathcal{A}:f|_{\Gamma_-}\rightarrow f|_{\Gamma_+},$$ which maps the incoming Dirichlet type boundary condition into the outgoing one. In particular, $\mathcal{A}$ is a bounded operator from $L^1(\Gamma_-,d\xi)$ to $L^1(\Gamma_+,d\xi)$, as shown in \cite{CS98}.

In the diffusion regime, we consider optical parameters $(\sigma_{\Kn}, \Kn^{-1}\sigma_s)$, instead of $(\sigma, k)$ as in section \ref{mainresult}.
Assume that $(\sigma_{\Kn}, \Kn^{-1}\sigma_s)$ is also admissible. From \cite{CS1, CS2, CS98}, it was shown that 
the albedo operator $\mathcal{A}$ is bounded from $L^1(\Gamma_-,d\xi)$ to $L^1(\Gamma_+,d\xi)$ equipped with the kernel   $\alpha(x,v,x',v')=(\mathcal{A}_1+\mathcal{A}_2+\mathcal{A}_3)(x,v,x',v')$, 
where  
\begin{align}
\mathcal{A}_1(x,v,x',v') &= e^{ -\int^{\tau_-(x,v)}_0 \sigma_{\Kn}(x-tv) dt} \delta_{x-\tau_-(x,v)v}(x') \delta(v-v'),\label{singular_A1}\\
\mathcal{A}_2(x,v,x',v') &=-\int^{\tau_-(x,v)}_0e^{ -\int^\eta_0 \sigma_{\Kn}(x-tv) dt -\int^{\tau_-(x-\eta v,v')}_0 \sigma_{\Kn}(x-\eta v-tv') dt  }\notag\\
& \hskip3.5cm k(x-\eta v,v',v) \delta_{x-\eta v-\tau_-(x-\eta v,v')v'}(x')d\eta, \label{kernel_A2}
\end{align}
and
\begin{align}\label{kernel_A3}
|n_{x'}\cdot v'|^{-1} & \mathcal{A}_3(x,v,x',v') \in L^\infty(\Gamma_-;L^1 (\Gamma_+,d\xi)).
\end{align} 
Here $d\xi(x,v)$ is defined in~\eqref{eqn:bdy_measure}. In addition, $\delta(x)$ is the delta function on $\R^n$ and $\delta_y(x)$ is the delta function on $\p\Omega$ defined by $(\delta_y,h)=h(y)$ for $h\in C^\infty_c(\R^n)$.
The travel time is denoted by $\tau_{\pm}(x,v)=\min\{t\geq 0: (x\pm tv,v)\in \Gamma_{\pm}\}$.

Notice that the kernel $\mathcal{A}_1$ is a singular distribution supported on the surface $x'=x-\tau_-(x,v)v $ and $v=v'$. One can apply the different degrees of singularities of $\mathcal{A}_j$, $j=1,2,3$, to distinguish $\mathcal{A}_1$ from the whole kernel $\alpha$. Then the information of $\sigma_{\Kn}$ can be extracted from $\mathcal{A}_1$. More precisely, the X-ray transform (defined in \eqref{def_Xray}) of $\sigma_{\Kn}$ will be first recovered from $\mathcal{A}_1$. Moreover, based on this, we could derive the stability estimate for $\sigma_{\Kn}$ with explicit coefficients' dependence on $\Kn$, see section \ref{stability_est}. 

We state the following lemma \ref{CSlemma} and lemma \ref{Wlemma}. Their proof can be found in \cite{CS98} and in \cite{Wang1999}, respectively. 
\begin{lemma}\label{CSlemma} Let $f\in L^1(\Omega\times V)$. Then
	\begin{align*}
	\int_{\Omega\times V} f(x,v)dxdv = \int_{\Gamma_\mp}\int^{\tau_\pm(x',v)}_0f(x'\pm tv,v)dtd\xi(x',v)\,.
	\end{align*} 
\end{lemma}
\begin{lemma}\label{Wlemma} Let $f\in L^1(\Gamma_-,d\xi)$. Then
	\begin{align*}
	\int_{\Gamma_+} f(x-\tau_-(x,v)v,v) d\xi(x,v)= \int_{\Gamma_-} f(x',v)d\xi(x',v)\,.
	\end{align*} 
\end{lemma}
These identities will play an crucial role in the derivation of the stability estimate for the optical parameters. In particular, lemma \ref{Wlemma} implies that integrals on the space $\Gamma_+$ and on the space $\Gamma_-$ are the same under suitable change of variables. It will be applied in lemma \ref{A1est} in order to transform the integral over the outgoing space $\Gamma_+$ into the integral over the incoming space $\Gamma_-$ such that the X-ray transform of $\sigma_{\Kn}$ can be recovered from the first kernel $\mathcal{A}_1$. 
As for lemma \ref{CSlemma}, it gives a way to compute the integral over $\Omega\times V$ by using the line and surface integrals, and vice versa.

\section{Analysis of the Knudsen number}\label{stability_est}
In this section, we study the stability estimate of the absorption and the scattering coefficient and investigate their dependence on the Kundsen number. 
We start by analyzing the total absorption coefficient $\sigma_{\Kn}=\Kn\sigma_a+\Kn^{-1}\sigma_s$ in \eqref{sigma_a}. The analysis technique is based on \cite{CS3, CS98, Wang1999} with suitable adjustments to our setting. 
\subsection{Stability estimate of the total absorption coefficient $\sigma_{\Kn}$}

Assume that the function $\psi \in C^\infty_0(\R^n)$ satisfies $0\leq \psi \leq 1$, $\psi (0)=1$ and $\int \psi  dx=1$.  
Let $(x_0',v_0')\in \Gamma_-$ and $\varepsilon>0$. We denote the functions	
$$
    \psi_{v_0'}^\varepsilon (v) = \varepsilon^{-n} \psi\LC {v-v'_0\over \varepsilon}\RC.
$$
We also choose functions $\phi_{x_0'}^\varepsilon$ in spatial dimensions such that $0\leq \phi_{x_0'}^\varepsilon(x)\in C^\infty_0(\R^n)$ and  $\supp\phi_{x_0'}^\varepsilon(x )\in B^\varepsilon(x_0')\cap \p\Omega$. Moreover, $\int_{\p\Omega} \phi_{x_0'}^\varepsilon(x)d\mu(x)=1$ and
$$
    \lim_{\varepsilon\rightarrow 0} \int_{\p\Omega}f(x)\phi^\varepsilon_{x_0'}(x)d\mu(x) = f(x_0')
$$
for any function $f$ in $C^0(\p\Omega).$
For $(x_0',v_0')\in \Gamma_-$, if $\varepsilon>0$ is sufficiently small, then one has
$$
    \supp \phi^\varepsilon_{x_0'}(x')\times \supp \psi^\varepsilon_{v_0'}(v')\subset \Gamma_-.
$$
Next, we denote the smooth cut off function $f^\varepsilon_{x_0',v_0'}$ on $\p\Omega$ and the velocity space by
\begin{align}\label{f}
    f^\varepsilon_{x_0',v_0'}( x',v')=|n_{x'}\cdot v'|^{-1} \phi^\varepsilon_{x_0'}(x')\psi^\varepsilon_{v_0'}(v') .
\end{align}
From a direct computation, one has $f^\varepsilon_{x_0',v_0'}\in L^1(\Gamma_-,d\xi)$. In particular, 
\begin{equation}\label{eqn:int_f}
\|f^\varepsilon_{x_0',v_0'}\|_{L^1(\Gamma_-,d\xi)} = \int_{\Gamma_-} |f^\varepsilon_{x_0',v_0'}(x',v')|   |n_{x'}\cdot v'|d\mu(x')dv'=1.
\end{equation}
Before studying the kernel, for $(x_0',v_0')\in \Gamma_-$, we define another cut off function on $\Gamma_+$ by
$$
\tilde{\chi}^\varepsilon(x,v):=\chi^\varepsilon(x-\tau_-(x,v)v,v) = \chi^{1,\varepsilon}_{x_0'}(x-\tau_-(x,v)v)\chi^{2,\varepsilon}_{v_0'}(v) 
$$
for any $(x ,v )$ in $\Gamma_+$, where $\chi^{1,\varepsilon}$ and $\chi^{2,\varepsilon}$ satisfy
\begin{align*}
     \left\{ \begin{array}{cl}
        \chi^{1,\varepsilon}(x) =1 &  \hbox{in }B^\varepsilon(x_0')\cap \p\Omega,\\
        \chi^{1,\varepsilon} (x) =0 & \hbox{in }\R^n  \setminus (B^\varepsilon(x_0')\cap \p\Omega),\\
            \end{array}  \right.
\end{align*}
and 
\begin{align*}
     \left\{ \begin{array}{cl}
        \chi^{2,\varepsilon}(v) =1 &  \hbox{in }  \supp \psi^\varepsilon_{v_0'}(v),\\
        \chi^{2,\varepsilon}(v)=0 & \hbox{in } V \setminus  \supp \psi^\varepsilon_{v_0'}(v).\\
            \end{array}  \right.
\end{align*}

The main goal of this section is to extract the information of $\sigma_{\Kn}-\tilde{\sigma}_{\Kn}$ from the measurements $\mathcal{A}-\tilde{\mathcal{A}}$. Let $\alpha$ and $\tilde{\alpha}$ be the distribution kernel for $\mathcal{A}$ and $\tilde{\mathcal{A}}$, respectively. We apply the cut off function on the albedo operator and then estimate the function
\begin{align*}
    &\tilde{\chi}^\varepsilon (\mathcal{A}-\tilde{\mathcal{A}}) f_{x_0',v_0'}^\varepsilon(x,v) \\
    &= \tilde{\chi}^\varepsilon(x,v)\int_{\Gamma_-} (\alpha(x,v,x',v')-\tilde{\alpha}  (x,v,x',v')) f_{x_0',v_0'}^\varepsilon(x',v')d\mu(x')dv',
\end{align*}
for $(x,v)\in \Gamma_+$.
In particular, we will estimate each term in the right hand side of \eqref{A3} which corresponds to $\mathcal{A}_j$, $j=1,2,3$, respectively:
\begin{align}\label{A3}
\|\tilde{\chi}^\varepsilon (\mathcal{A}-\tilde{\mathcal{A}}) f_{x_0',v_0'}^\varepsilon \|_{L^1(\Gamma_+,d\xi)} = \left\|\sum_{j=1}^3 \tilde{\chi}^\varepsilon \int_{\Gamma_-}(\mathcal{A}_j-\tilde{\mathcal{A}}_j) f_{x_0',v_0'}^\varepsilon(x',v') d\mu(x')dv'\right\|_{L^1(\Gamma_+,d\xi)}.
\end{align} 
We start by considering the following estimate.
\begin{lemma}\label{A1est} For $\varepsilon>0$, $f^\varepsilon_{x_0',v_0'}$ is defined as in \eqref{f}. Then
\begin{align*}
    &\lim_{\varepsilon\rightarrow 0}\LN \tilde\chi^\varepsilon \int_{\Gamma_-}  (\mathcal{A}_1-\tilde{\mathcal{A}}_1) f_{x_0',v_0'}^\varepsilon(x',v')d\mu(x')dv'\RN_{L^1(\Gamma_+,d\xi)} \\
    &= \LV e^{-\int_{\R} \sigma_{\Kn}(x'_0+tv_0') dt} -e^{-\int_{\R} \tilde{\sigma}_{\Kn}(x_0'+tv_0') dt}\RV.
\end{align*}
\end{lemma}
\begin{proof}
From the definition of the kernels $A_1$ and $\tilde{A}_1$, one has
\begin{align*}
   & \int_{\Gamma_-} (\mathcal{A}_1-\tilde{\mathcal{A}}_1)(x,v,x',v')f_{x_0',v_0'}^\varepsilon(x',v')d\mu(x') dv'\\
    &= \LC e^{-\int^{\tau_-(x,v)}_0 \sigma_{\Kn}(x-tv) dt} -e^{-\int^{\tau_-(x,v)}_0 \tilde{\sigma}_{\Kn}(x-tv) dt} \RC f^\varepsilon_{x_0',v_0'}(x -\tau_-(x,v)v, v).
\end{align*}  
Since $\sigma_{\Kn}$ and $\tilde{\sigma}_{\Kn}$ are supported in $\Omega$, the integration range can be extended to $\R$. 
Thus, we obtain
\begin{align}\label{A1}
   &\LN\tilde\chi^\varepsilon \int_{\Gamma_-}  (\mathcal{A}_1-\tilde{\mathcal{A}}_1) f_{x_0',v_0'}^\varepsilon(x',v')d\mu(x')dv'\RN_{L^1(\Gamma_+,d\xi)} \notag \\
   &=   \int_{\Gamma_+}\tilde{\chi}^\varepsilon(x,v)  \LV e^{-\int_{\R} \sigma_{\Kn}(x-\tau_-(x,v)v+tv) dt} -e^{-\int_{\R} \tilde{\sigma}_{\Kn}(x-\tau_-(x,v)v+tv) dt}\RV  \notag\\
   &\quad\quad  f^\varepsilon_{x_0',v_0'}(x -\tau_-(x,v)v, v)d\xi(x,v) \notag\\
  & =  \int_{\Gamma_-}  \chi^\varepsilon(y,v)\LV e^{-\int_{\R} \sigma_{\Kn}(y+tv) dt} -e^{-\int_{\R} \tilde{\sigma}_{\Kn}(y+tv) dt} \RV f^\varepsilon_{x_0',v_0'}(y, v)d\xi(y,v) .
\end{align}  
Here the last identity holds by applying lemma \ref{Wlemma} and $f^\varepsilon_{x_0',v_0'} \in L^1(\Gamma_-,d\xi)$. 
Moreover, from the definition of $\chi^\varepsilon$, one has $\chi^\varepsilon$ is compactly supported and $\chi^\varepsilon=1$ in $(B^\varepsilon(x_0')\cap\p\Omega) \times \supp\psi^\varepsilon_{v_0'}$. We apply the properties of the function $f^\varepsilon_{x_0',v_0'}$ with \eqref{eqn:int_f} and then, by taking the limit $\varepsilon\rightarrow 0$ on the identity \eqref{A1}, we conclude that
\begin{align*}
  &\LN \tilde\chi^\varepsilon\int_{\Gamma_-}  (\mathcal{A}_1-\tilde{\mathcal{A}}_1) f_{x_0',v_0'}^\varepsilon(x',v')d\mu(x')dv'\RN_{L^1(\Gamma_+,d\xi)} \\
  &=\int_{\Gamma_-}  \LV e^{-\int_{\R} \sigma_{\Kn}(y+tv) dt} -e^{-\int_{\R} \tilde{\sigma}_{\Kn}(y+tv) dt} \RV f^\varepsilon_{x_0',v_0'}(y,v)d\xi(y,v)\\ &\rightarrow  \LV e^{-\int_{\R} \sigma_{\Kn}(x_0'+tv_0') dt} -e^{-\int_{\R} \tilde{\sigma}_{\Kn}(x_0'+tv_0') dt} \RV.
\end{align*} 
This finishes the proof.
\end{proof}

For the remaining two terms in \eqref{A3}, we have the following identities.
\begin{lemma}\label{A23} For $\varepsilon>0$, $f^\varepsilon_{x_0',v_0'}$ is defined in \eqref{f}. Then
\begin{align*}
    \lim_{\varepsilon\rightarrow 0} \LN \tilde\chi^\varepsilon\int_{\Gamma_-}  (\mathcal{A}_j-\tilde{\mathcal{A}}_j) f_{x_0',v_0'}^\varepsilon(x',v')d\mu(x')dv'\RN_{L^1(\Gamma_+,d\xi)} = 0,\ \ j=2,3.
\end{align*}
\end{lemma}
\begin{proof}
We first study the case $j=2$. Based on the definition of the kernel $\mathcal{A}_2$ in \eqref{kernel_A2}, the delta function $\delta_{x-\eta v-\tau_-(x-\eta v,v')v'}(x')$ acts on the function $f_{x_0',v_0'}^\varepsilon(x',v')$ will take the value $f^\varepsilon_{x_0',v_0'}(x -\tau_-(x-\eta v,v')v'-\eta v, v')$. Then one has
    \begin{align*} 
       & \LN \tilde{\chi}^\varepsilon\int_{\Gamma_-}  (\mathcal{A}_2-\tilde{\mathcal{A}}_2) f_{x_0',v_0'}^\varepsilon(x',v')d\mu(x')dv' \RN_{L^1(\Gamma_+,d\xi)}  \notag\\
       &= \int_{\Gamma_+}\tilde\chi^\varepsilon(x,v) \Big| \int_V\int_0^{\tau_-(x,v)} (\Gamma-\tilde{\Gamma})(x,v,x',v')\notag \\
       &\quad\quad f^\varepsilon_{x_0',v_0'}(x -\tau_-(x-\eta v,v')v'-\eta v, v') d\eta dv' \Big| d\xi(x,v) \,,
    \end{align*}  
where we denote 
$$ 
\Gamma (x,v,x',v') = \Kn^{-1} e^{-\int^\eta_0 \sigma_{\Kn}(x-tv)dt - \int^{\tau_-(x-\eta v,v')}_0 \sigma_{\Kn}(x-\eta v-tv')dt} \sigma_s(x- \eta v),
$$ 
and 
$$
\tilde\Gamma(x,v,x',v') = \Kn^{-1} e^{-\int^\eta_0 \tilde\sigma_{\Kn}(x-tv)dt - \int^{\tau_-(x-\eta v,v')}_0 \tilde\sigma_{\Kn}(x-\eta v-tv')dt} \tilde\sigma_s(x-\eta v).
$$
Note that since $\sigma_s$ and $\tilde\sigma_s$ are nonnegative, it gives $$|\Gamma|\leq \Kn^{-1} \sigma_s(x- \eta v), \ \ \ |\tilde\Gamma|\leq \Kn^{-1} \tilde\sigma_s(x-\eta v) .$$
We interchange the integration order by using Fubini's theorem, and lemma \ref{CSlemma},  
we have 
  \begin{align}\label{id3}
       &\LN \tilde{\chi}^\varepsilon\int_{\Gamma_-} (\mathcal{A}_2-\tilde{\mathcal{A}}_2) f_{x_0',v_0'}^\varepsilon(x',v')d\mu(x')dv'\RN_{L^1(\Gamma_+,d\xi)}  \notag\\
       &\leq    \int_{\Gamma_+} \int_V\int_0^{\tau_-(x,v)}    \chi^{2,\varepsilon}( v) \Kn^{-1} (\sigma_s+\tilde{\sigma}_s)(x-\eta v)  \notag\\
       &\quad\quad f^\varepsilon_{x_0',v_0'}(x -\eta v -\tau_-(x-\eta v,v')v', v') d\eta dv' d\xi(x,v)  \notag \\
         & =  \int_{V} \int_{\Omega\times V}  \chi^{2,\varepsilon}( v) \Kn^{-1} (\sigma_s+\tilde{\sigma}_s)(x)  
           f^\varepsilon_{x_0',v_0'}(x -\tau_-(x,v')v', v') dxdv dv' .
\end{align}
Using lemma \ref{CSlemma} again and the bounded condition for $v$ described in \eqref{eqn:v_bound}, it leads to
\begin{align}\label{A2:inquality} 
        &\int_{V} \chi^{2,\varepsilon}( v) \Big( \int_{\Omega\times V}    \Kn^{-1} (\sigma_s+\tilde{\sigma}_s)(x)  
        f^\varepsilon_{x_0',v_0'}(x -\tau_-(x,v')v', v') dxdv'\Big) dv \notag\\ 
        &=  \int_{V} \chi^{2,\varepsilon}( v) \Big( \int_{\Gamma_-}\int_0^{\tau_+(x',v')}    \Kn^{-1} (\sigma_s+\tilde{\sigma}_s)(x'+t v') 
        f^\varepsilon_{x_0',v_0'}(x', v') dt d\xi(x',v')\Big) dv \notag\\
        &\leq  {\diam(\Omega) \over M_1}\Kn^{-1}\|(\sigma_s+\tilde{\sigma}_s)\|_{L^\infty(\Omega)} \Big(  \int_{\Gamma_-} f^\varepsilon_{x_0',v_0'}(x', v')  d\xi(x',v')\Big) \Big( \int_{\supp \psi^\varepsilon_{v_0'}(v)}dv\Big).
\end{align}  
The last component of the above equality has the measure of $\supp \psi^\varepsilon_{v_0'}(v)$ goes to $0$ when $\varepsilon\rightarrow 0$. It implies that the right hand side of \eqref{A2:inquality} converges to zero, and thus we obtain the conclusion of the Lemma for $j=2$.

Now we will turn to the term with kernel $\mathcal{A}_3$ and $\tilde{\mathcal{A}}_3$. From \eqref{kernel_A3}, they satisfy $$|n_{x'}\cdot v'|^{-1}\mathcal{A}_3,\ |n_{x'}\cdot v'|^{-1}\tilde{\mathcal{A}}_3 \in L^\infty(\Gamma_-;L^1(\Gamma_+,d\xi)).$$ Thus, the limit of 
  \begin{align}\label{id2}
       &\LN \tilde\chi^\varepsilon \int_{\Gamma_-}(\mathcal{A}_3-\tilde{\mathcal{A}}_3) f_{x_0',v_0'}^\varepsilon(x',v') d\mu(x')dv'\RN_{L^1(\Gamma_+,d\xi)}  \notag\\
       &=   \int_{\Gamma_+}\tilde\chi^\varepsilon(x,v) \LV\int_{\Gamma_-} (\mathcal{A}_3-\tilde{\mathcal{A}}_3)(x,v,x',v') f^\varepsilon_{x_0',v_0'}(x',v') d\mu(x')dv'\RV d\xi(x,v)  \notag\\
       &\leq  \int_{\Gamma_+} \tilde\chi^\varepsilon(x,v) \int_{\Gamma_-} |n_{x'}\cdot v'|^{-1}|(\mathcal{A}_3-\tilde{\mathcal{A}}_3)(x,v,x',v')|f^\varepsilon_{x_0',v_0'}(x',v')d\xi(x',v') d\xi(x,v) \notag\\
       &\leq  \int_{\Gamma_+} \tilde\chi^\varepsilon(x,v)  \sup_{(x',v')\in \Gamma_-}|n_{x'}\cdot v'|^{-1}|(\mathcal{A}_3-\tilde{\mathcal{A}}_3)(x,v,x',v')|  d\xi(x,v) 
    \end{align}  
goes to $0$ as $\varepsilon\rightarrow 0$ by applying dominated convergence theorem and the fact that the measure of support of $\tilde{\chi}^\varepsilon$ converges to zero as $\varepsilon \rightarrow 0$. This completes the proof of this lemma.
\end{proof}

From the equation \eqref{A3}, we have the estimate for the term containing $\mathcal{A}_1-\tilde{\mathcal{A}_1}$:
\begin{align*}
    &  \LN \tilde\chi^\varepsilon \int_{\Gamma_-}   (\mathcal{A}_1-\tilde{\mathcal{A}}_1) f_{x_0',v_0'}^\varepsilon d\mu(x')dv'\RN_{L^1(\Gamma_+,d\xi)} \\
   &\leq    \|  \tilde{\chi}^\varepsilon (\mathcal{A}-\tilde{\mathcal{A}}) f_{x_0',v_0'}^\varepsilon \|_{L^1(\Gamma_+,d\xi)} +\sum_{j=2,3} \LN \tilde\chi^\varepsilon \int_{\Gamma_-}   (\mathcal{A}_j-\tilde{\mathcal{A}}_j) f_{x_0',v_0'}^\varepsilon d\mu(x')dv'\RN_{L^1(\Gamma_+,d\xi)}.
\end{align*}
Let $\varepsilon$ go to $0$, then by lemma \ref{A1est} and lemma \ref{A23}, this leads to
\begin{align*}
\LV e^{-\int_{\R} \sigma_{\Kn}(x_0'+tv_0') dt} -e^{-\int_{\R} \tilde{\sigma}_{\Kn}(x_0'+tv_0') dt} \RV
   &\leq  \lim_{\varepsilon\rightarrow 0}   \|\tilde\chi^\varepsilon(\mathcal{A}-\tilde{\mathcal{A}}) f_{x_0',v_0'}^\varepsilon \|_{L^1(\Gamma_+,d\xi)} \\
& \leq \| (\mathcal{A}-\tilde{\mathcal{A}}) f_{x_0',v_0'}^\varepsilon \|_{L^1(\Gamma_+,d\xi)}  \\
& \leq \| \mathcal{A}-\tilde{\mathcal{A}} \|_{*}\| f_{x_0',v_0'}^\varepsilon \|_{L^1(\Gamma_-,d\xi)}\\
& =\| \mathcal{A}-\tilde{\mathcal{A}}\|_{*} ,
\end{align*}
where we use the fact that $\| f_{x_0',v_0'}^\varepsilon\|_{L^1(\Gamma_-,d\xi)}=1$ in the last identity.
Therefore, applying the mean value theorem on the left hand side of the above inequalities, it follows that
\begin{align}\label{inequality}
\| \mathcal{A}-\tilde{\mathcal{A}}\|_{*} 
&\geq e^{-\beta_{\Kn}	} \LV \int_{\R} \sigma_{\Kn}(x_0'+tv_0') - \tilde{\sigma}_{\Kn}(x_0'+tv_0') dt\RV,
\end{align}
where one can deduce the constant bound $\beta_{\Kn}$ by using once again the boundedness of $v$ in~\eqref{eqn:v_bound} as follows:
\begin{align}\label{beta_kn}
\beta_{\Kn}={\diam(\Omega) } M_1^{-1} \LC \Kn (\| \sigma_a\|_{L^\infty}+\|\tilde{\sigma}_a\|_{L^\infty}) +\Kn^{-1} (\| \sigma_s\|_{L^\infty}+\|\tilde{\sigma}_s\|_{L^\infty})\RC.
\end{align}

Before going further, let us introduce some notations. We denote the set which consists the unit vectors in $\R^n$ by $$\mathbb{S}^{n-1}=\{x\in\R^n:\ |x|=1\}.$$ 
In $X$-ray tomography, a ray goes through the point $x\in\R^n$ and has the direction $\omega\in \mathbb{S}^{n-1}$. Integrating over this ray leads to the X-ray transform $Xf$ of $f$, which is defined as
\begin{align}\label{def_Xray}
    (Xf)(x,\omega) = \int_{\R} f(x+s\omega) ds 
\end{align}
for every $\omega\in \mathbb{S}^{n-1}$. 
We denote a function $g$ in the space $T\mathbb{S}^{n-1}$ by
$$
 \|g\|_{L^2(T\mathbb{S}^{n-1})}^2 = \int_{\mathbb{S}^{n-1}} \|g(\cdot, \omega)\|^2_{L^2(T)}d\omega,
$$
where $T=\{(x,\omega)\in \R^n\times \mathbb{S}^{n-1}:\ x\cdot \omega=0 \}$.
We also denote the space $$\p\Omega\times \mathbb{S}^{n-1}_-=\{(x,\omega)\in \p\Omega\times \mathbb{S}^{n-1}:\ n_{x}\cdot \omega<0\}.$$
In particular, one can deduce that 
there is a constant $C_0>0$ such that
\begin{align}\label{X_1}
    \|Xf\|_{L^2(T\mathbb{S}^{n-1})}\leq C_0 \|Xf\|_{L^\infty(\p\Omega\times \mathbb{S}^{n-1}_-)}
\end{align}
for all functions $Xf$ in $L^\infty(\p\Omega\times \mathbb{S}^{n-1}_-)$. 
Further, from Theorem 3.1 in \cite{LN1983}, for any function $f\in H^{-1/2}(\Omega)$ with compact support in $\Omega$, then there exists a constant $C_1>0$ such that 
\begin{align}\label{X_2}
    \|f\|_{H^{-1/2}(\Omega)}\leq C_1\|Xf\|_{L^2(T\mathbb{S}^{n-1})}.
\end{align}
Combining \eqref{X_1} and $\eqref{X_2}$, this leads to 
\begin{align}\label{X_3}
    \|f\|_{H^{-1/2}(\Omega)}\leq C_0C_1\|Xf\|_{L^\infty(\p\Omega\times \mathbb{S}^{n-1}_-)}.
\end{align}
We use it to show the following stability estimate.
\begin{proposition}\label{prop:H_onehalf}
We denote $\hat{v}_0'=v_0'/|v_0'|$ to be the unit vector. Then
\begin{align*} 
     \| \sigma_{\Kn}-\tilde{\sigma}_{\Kn} \|_{H^{-1/2}(\Omega)}
     & \leq C_0C_1\|X(\sigma_{\Kn} -\tilde{\sigma}_{\Kn})\|_{L^\infty (\p\Omega\times \mathbb{S}^{n-1}_-)}\\
     &  \leq C_0C_1M_2 e^{\beta_{\Kn}} \| \mathcal{A}-\tilde{\mathcal{A}}\|_{*},
\end{align*}
where $\beta_{\Kn}>0$ is defined as in \eqref{beta_kn} and $M_2$ is the upper bound of $v$ stated in~\eqref{eqn:v_bound}.
\end{proposition} 
\begin{proof}
For any $(x_0',v_0')\in\Gamma_-$, by applying the change of variable $t \mapsto|v_0'|t$ in \eqref{inequality}, then we have
\begin{align*}
 \| \mathcal{A}-\tilde{\mathcal{A}}\|_{*}
 & \geq e^{-\beta_{\Kn}} |v_0'|^{-1} |X(\sigma_{\Kn} -\tilde{\sigma}_{\Kn})(x_0',\hat{v}_0')| . 
\end{align*}
The desired estimates follow by applying \eqref{X_3} to the above identity.
\end{proof}

\subsection{Proof of Theorem \ref{thm1.1}}
The norm used in Proposition~\ref{prop:H_onehalf} is rather weak. But since we assume a-priori that the media is in the function space $\mathcal{P}$ with higher regularity as defined in~\eqref{eqn:P_space}, interpolation formula could be used to lift the stability estimate to a stronger result. Recall the interpolation formula which states the existence of constant $C_2$ so that:
\begin{align}\label{interpolation}
\|u\|_{H^{s}(\R^n)}\leq C_2 \|u\|^\theta_{H^{s_1}(\R^n)} \|u\|^{1-\theta}_{H^{s_2}(\R^n)}\,,
\end{align}
for any $s_1<s_2$ and that $s=\theta s_1+(1-\theta)s_2$ with $0< \theta<1$, and that the constant purely depends on $C_2=C_2(n,s_1,s_2)$. One simply needs to choose a special set of $(s_1,s_2)$ to achieve the results in Theorem~\ref{thm1.1}.

\begin{proof}[\textit{Proof of Theorem \ref{thm1.1}:}]
For $r'>0$, let $0<r<r'$ and set $s=\frac{3}{2}+r$. For $s_1 = -\frac{1}{2}$ and $s_2 = \frac{3}{2}+r'$, we easily see that $s_1<s<s_2$ and there is a constant $\theta$ such that $s=\theta s_1+(1-\theta)s_2$. Using the interpolation formula~\eqref{interpolation}, we have
\begin{align*}
       \| \sigma_{\Kn}-\tilde{\sigma}_{\Kn} \|_{H^{3/2+r}} \leq C_2\| \sigma_{\Kn}-\tilde{\sigma}_{\Kn} \|_{H^{3/2 + r'}}^{1-\theta} \| \sigma_{\Kn}-\tilde{\sigma}_{\Kn} \|_{H^{-1/2}}^\theta\,.
       \end{align*}
From the hypothesis of Theorem \ref{thm1.1}, one has $\sigma_{a},\ \sigma_s,\ \tilde{\sigma}_a$, and $\tilde{\sigma}_s$ are in $\mathcal{P}$ in~\eqref{eqn:P_space}. Combining with Proposition~\ref{prop:H_onehalf}, the inequality could be further bounded by:
       \begin{align*}
       \| \sigma_{\Kn}-\tilde{\sigma}_{\Kn} \|^{1/\theta}_{H^{3/2+r}} &\leq  \left(2C_2 M_3^{1-\theta}\LC \Kn+ \Kn^{-1}\RC^{1-\theta}\right)^{1/\theta} \| \sigma_{\Kn}-\tilde{\sigma}_{\Kn} \|_{H^{-1/2}} \\
       &\leq\LC 2C_2 M_3^{1-\theta}\LC \Kn+ \Kn^{-1}\RC^{1-\theta}\RC^{1/\theta} C_0 C_1M_2  e^{\beta_{\Kn}}\|\mathcal{A}-\tilde{\mathcal{A}} \|_*.
\end{align*}

According to the definition of $\beta_{\Kn}$, one has
\begin{align*} 
\beta_{\Kn}\leq C(\Kn+\Kn^{-1}) 
\end{align*}
for some constant $C$ independent of $\Kn$. Then applying Sobolev imbedding theorem, we have the following $L^\infty$ estimate for $\sigma_{\Kn}-\tilde{\sigma}_{\Kn}$:
\begin{align}\label{totalsigma}
\| \sigma_{\Kn}-\tilde{\sigma}_{\Kn} \|_{L^\infty}^{1/\theta}\leq 
\LC 2C_2 M_3^{1-\theta}\LC \Kn+ \Kn^{-1}\RC^{1-\theta}\RC^{1/\theta}  C_0 C_1 C_3M_2  e^{C(\Kn+\Kn^{-1})}\|\mathcal{A}-\tilde{\mathcal{A}} \|_*,
\end{align}
where the constants $C_j$, $j=0\,,\cdots,3$, are independent of $\Kn$.

Finally, we are ready to show \eqref{sigma_s_equ_1_thm} and \eqref{sigma_a_equ_1_thm}. From the definition of $\sigma_{\Kn}$ and \eqref{totalsigma}, we have
\begin{align*}
 \Kn\|\sigma_a-\tilde{\sigma}_a\|_{L^\infty} \leq \Kn^{-1}\|\sigma_s-\tilde{\sigma}_s\|_{L^\infty} +  C\LC \Kn+ \Kn^{-1}\RC^{1-\theta}  e^{C\theta(\Kn+\Kn^{-1})} \|\mathcal{A}-\tilde{\mathcal{A}} \|_*^\theta,
\end{align*}
where $C$ independent of $\Kn$. In particular, we have the following estimate for $\sigma_a-\tilde{\sigma}_a$: 
\begin{align}\label{sigma_a_inq}
\|\sigma_a-\tilde{\sigma}_a\|_{L^\infty} \leq   \Kn^{-2}\|\sigma_s-\tilde{\sigma}_s\|_{L^\infty}  +  C\Kn^{-1}\LC \Kn+ \Kn^{-1}\RC^{1-\theta}  e^{C\theta(\Kn+\Kn^{-1})} \|\mathcal{A}-\tilde{\mathcal{A}} \|_*^\theta.
\end{align} 
On the other hand, one can also derive an estimate for $\sigma_s-\tilde{\sigma}_s$:
\begin{align}\label{sigma_s_equ}
\|\sigma_s-\tilde{\sigma}_s\|_{L^\infty} \leq  \Kn^2 \|\sigma_a-\tilde{\sigma}_a\|_{L^\infty} +  C\Kn \LC \Kn+ \Kn^{-1}\RC^{1-\theta}  e^{C\theta(\Kn+\Kn^{-1})} \|\mathcal{A}-\tilde{\mathcal{A}} \|_*^\theta.
\end{align}

Assume that $\|\mathcal{A}-\tilde{\mathcal{A}} \|_*<1$. Under the assumption that scattering dominates absorption in a region of interest, we consider the case:
\begin{align}
\Kn< \min\{|\log( \|\mathcal{A}-\tilde{\mathcal{A}} \|_*)|^{-\alpha},\ 1 \}
\end{align}
for some constant $\alpha>0$. Then from \eqref{sigma_s_equ} and from the hypothesis, $\sigma_a,\ \tilde{\sigma}_a\in\mathcal{P}$, it follows that
\begin{align}\label{sigma_s_equ_1}
\|\sigma_s-\tilde{\sigma}_s\|_{L^\infty} 
&\leq  C\Kn^2 M_3 +  C\Kn \LC \Kn+ \Kn^{-1}\RC^{1-\theta}  e^{C\theta \Kn^{-1}} \|\mathcal{A}-\tilde{\mathcal{A}} \|_*^\theta  \notag\\
&\leq C\Kn |\log( \|\mathcal{A}-\tilde{\mathcal{A}} \|_*)|^{-\alpha} + C\Kn \LC \Kn+ \Kn^{-1}\RC^{1-\theta}  e^{ C\theta{\Kn}^{-1}} \|\mathcal{A}-\tilde{\mathcal{A}} \|_*^\theta\,,
\end{align}

On the other hand, one can also use a similar argument to derive the following estimate based on \eqref{sigma_a_inq},
\begin{align}\label{sigma_a_inq_1}
\|\sigma_a-\tilde{\sigma}_a\|_{L^\infty} 
&\leq C\Kn^{-3} |\log( \|\mathcal{A}-\tilde{\mathcal{A}} \|_*)|^{-\alpha} + C\Kn ^{-1}\LC \Kn+ \Kn^{-1}\RC^{1-\theta}  e^{ C\theta{\Kn}^{-1}} \|\mathcal{A}-\tilde{\mathcal{A}} \|_*^\theta.
\end{align}
This completes the proof of the theorem.
\end{proof}

\vskip.4cm 
We mention the following observations about the stability estimates which are derived in the proof of Theorem \ref{thm1.1}:

\begin{remark}
We briefly discuss the derived estimates \eqref{sigma_a_inq} and \eqref{sigma_s_equ} with the results obtained in the paper \cite{CLW} here.
Assume that $\sigma_s=\tilde\sigma_s$ is given which is the setup in section 3.1 and 3.2 in \cite{CLW}. In this setting, we study the stability of the coefficient $\sigma_a$ depending on the Knudsen number $\Kn$. Then we obtain a similar result as in \cite{CLW} where the linearized inverse problem is considered. In particular, we conclude from \eqref{sigma_a_inq} that the difference of $\sigma_a-\tilde{\sigma}_a$ could become larger if $\Kn$ is decreasing. This also means that a smaller $\Kn$ leads to worse distinguishability of the absorption coefficient.

When $\sigma_a=\tilde{\sigma}_a$ is known, similar to the observation in section 4.3 in \cite{CLW}, we have from \eqref{sigma_s_equ} that the difference of $\sigma_s-\tilde\sigma_s$ might increase as $\Kn$ shrinks.
\end{remark}

\begin{remark}
We specifically mention the goal of the estimates \eqref{sigma_s_equ_1} and \eqref{sigma_a_inq_1}. In the zero limit of $\Kn$, the radiative transfer equation becomes the diffusion equation, whose Dirichlet-to-Neumann map is shown to reconstruct the media with logarithmic instability. This is reflected in the corollary above as well. When $\Kn$ is small enough, the $\log( \|\mathcal{A}-\tilde{\mathcal{A}} \|_*)$ term becomes more and more dominant, leading to log-type illposedness.
\end{remark}





 \section{Numerics}\label{numerics}
We present numerical evidence in this section. We utilize a simpler model in $2D$:
\begin{equation}\label{eqn:rte_num}
v\cdot\nabla f = \cos\theta\partial_xf+\sin\theta\partial_yf = \frac{\sigma_s}{\Kn}\left(\langle f\rangle-f\right),
\end{equation}
where $\langle f\rangle_v = \int f\rd{\theta}$ with $\rd{\theta}$ is normalized. This is the critical case in the sense that the effective absorption is set to be zero. To demonstrate stability, we set two sets of media to be:
\begin{equation*}
\sigma_s = 1\quad \text{for}\,(x,y)\in\Omega\,;\quad \tilde{\sigma}_s = \begin{cases}1\quad &\text{for}\,(x,y)\in B\\1+z\quad &\text{for}\,(x,y)\in\Omega\backslash B\end{cases}\,,
\end{equation*}
with $\Omega =[0,0.6]^2$ and $B$ being a ball centered at $(0.3\,,0.3)$ with radius $0.2$. Obviously $\|\sigma_s-\tilde{\sigma}_s\|_{L^\infty} = z$, and in computation we choose $z$ to be $0.1\times \{1, 1/2, 1/4, 1/8\}$.

Denote the associated albedo operator $\mathcal{A}$ and $\tilde{\mathcal{A}}$, respectively, that map the incoming data to the outgoing data, and also denote $\mathcal{A}_1$ and $\tilde{\mathcal{A}}_1$ the leading order expansion of the albedo operator, as defined in~\eqref{singular_A1}. We will show numerical evidence from three aspects:
\begin{enumerate}
\item $\|\mathcal{A}_1\|_*$ decays exponentially fast with respect to $\Kn$;
\item For fixed $\Kn$, $\|\mathcal{A}-\tilde{\mathcal{A}}\|_\ast$ grows in a Lipschitz manner with respect to $z=\|\sigma_s-\tilde{\sigma}_s\|_{L^\infty}$;
\item For fixed $z$, $\|\mathcal{A}-\tilde{\mathcal{A}}\|_\ast$ blows up exponentially fast with respect to $\Kn$.
\end{enumerate}
Through the entire computation, we use $\mathrm{d}x = 0.025$ and $\mathrm{d}\theta = 2\pi/24$ to resolve all possible small scales. To obtain the operator norm, we need to numerically exhaust all possible boundary conditions. Term the collection of discrete boundary conditions $\mathcal{S}$. According to the numerical operator norm:
\begin{equation}\label{eqn:num_norm}
\|\mathcal{A}_1\|_\ast = \text{sup}_{\phi\in\mathcal{S}}\frac{\|\mathcal{A}_1\phi\|_{L^1(\Gamma_+)}}{\|\phi\|_{L^1(\Gamma_-)}}\,,\quad\text{and}\quad 
\|\mathcal{A}-\mathcal{A}_1\|_\ast = \text{sup}_{\phi\in\mathcal{S}}\frac{\|\mathcal{A}\phi-\mathcal{A}_1\phi\|_{L^1(\Gamma_+)}}{\|\phi\|_{L^1(\Gamma_-)}}\,,
\end{equation}
and to obtain $\mathcal{A}_1\phi$, we numerically solve:
\begin{equation}\label{eqn:x_ray_num}
v\cdot\nabla_xf = -\frac{1}{\Kn}\sigma_sf\,,\quad\text{with}\quad f|_{\Gamma_-}=\phi\,,
\end{equation}
and then confine the solution on $\Gamma_+$:
\begin{equation}
\mathcal{A}_1\phi = f|_{\Gamma+}\,.
\end{equation}
Similarly one can obtain $\mathcal{A}\phi$ by replacing~\eqref{eqn:x_ray_num} with~\eqref{eqn:rte_num}.

\noindent\textbf{Exponential decay in $\Kn$ of $\mathcal{A}_1$:}\\
In the first experiment, we set $z=0.1$ and let $\Kn = 2^k$ with $k$ changes from $1$ to $-3$ in the radiative transfer equation. We also evaluate $\|\mathcal{A}_1\|_\ast$ using~\eqref{eqn:num_norm}. Numerically we observe that $\mathcal{A}_1$'s operator norm decays with respect to $1/\Kn$, as seen in Figure~\ref{fig:A_A1}. The computation suggests that:
\begin{equation*}
\|\mathcal{A}_1\|_\ast\sim e^{-\frac{0.1}{\Kn}}\,,
\end{equation*}
In the zero limit of $\Kn$, the operator norm is extremely small. We also numerically evaluate $\mathcal{A}-\mathcal{A}_1$'s operator norm and study its dependence on $\Kn$. The numerical evidence shows that as $\Kn$ shrinks to zero, the discrepancy between $\mathcal{A}$ and $\mathcal{A}_1$ grows, which agrees with the observation made in~\cite{Bal_transport, Stefanov_2003}. We emphasize that $\mathcal{A}_1$ contains the most singular information in $\mathcal{A}$, by separating which one is able to recover the absorption coefficients ($\sigma_s$ here). In the zero limit of $\Kn$, $\mathcal{A}$ and $\mathcal{A}_1$ have big discrepancy, meaning $\mathcal{A}_1$ has very limited contribution in $\mathcal{A}$. This could potentially make the separation hard, leading to bad reconstruction.
\begin{figure}[htb]
	\includegraphics[width = 0.4\textwidth, height = 2.3in]{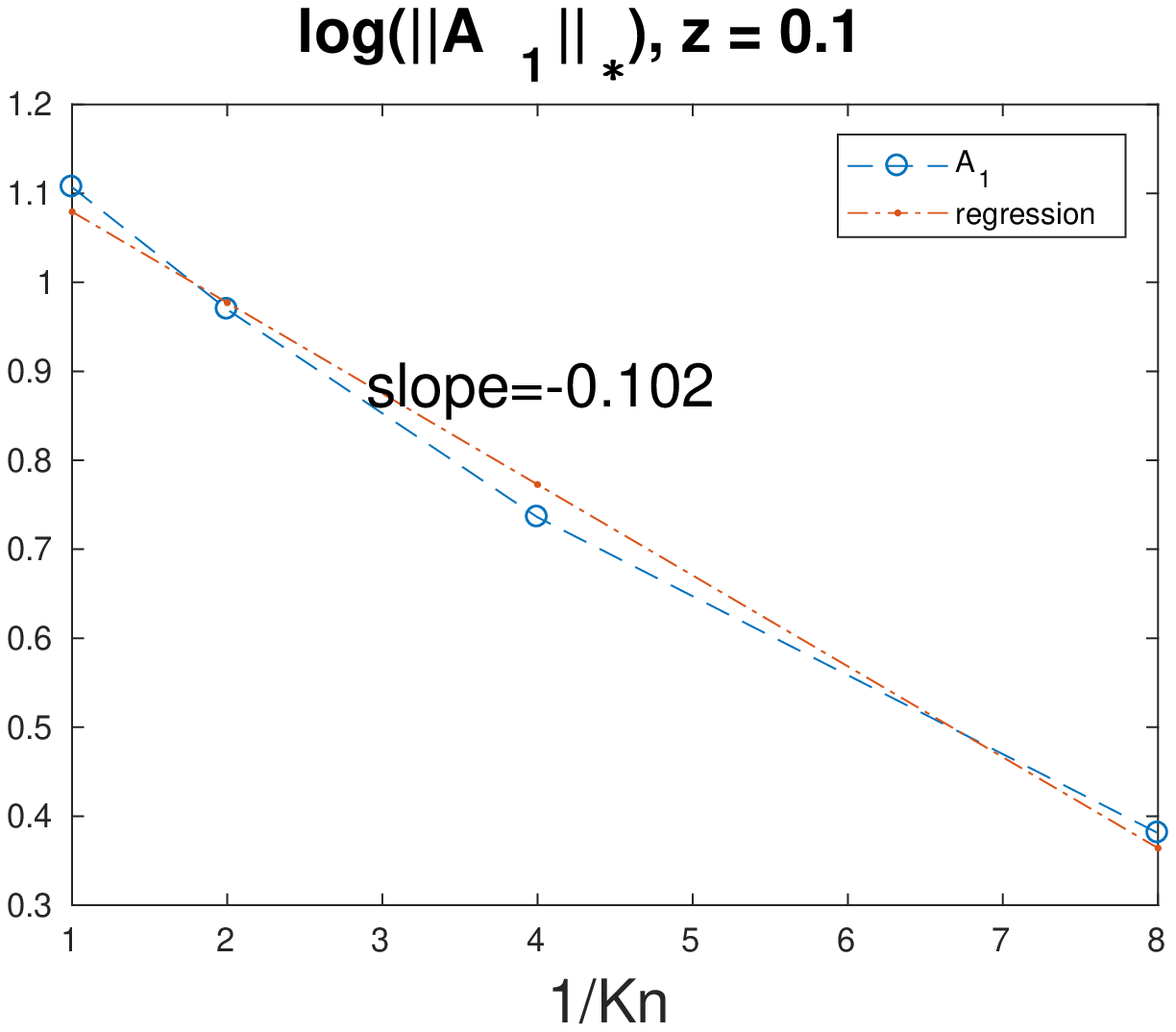}
	\includegraphics[width = 0.4\textwidth, height = 2.3in]{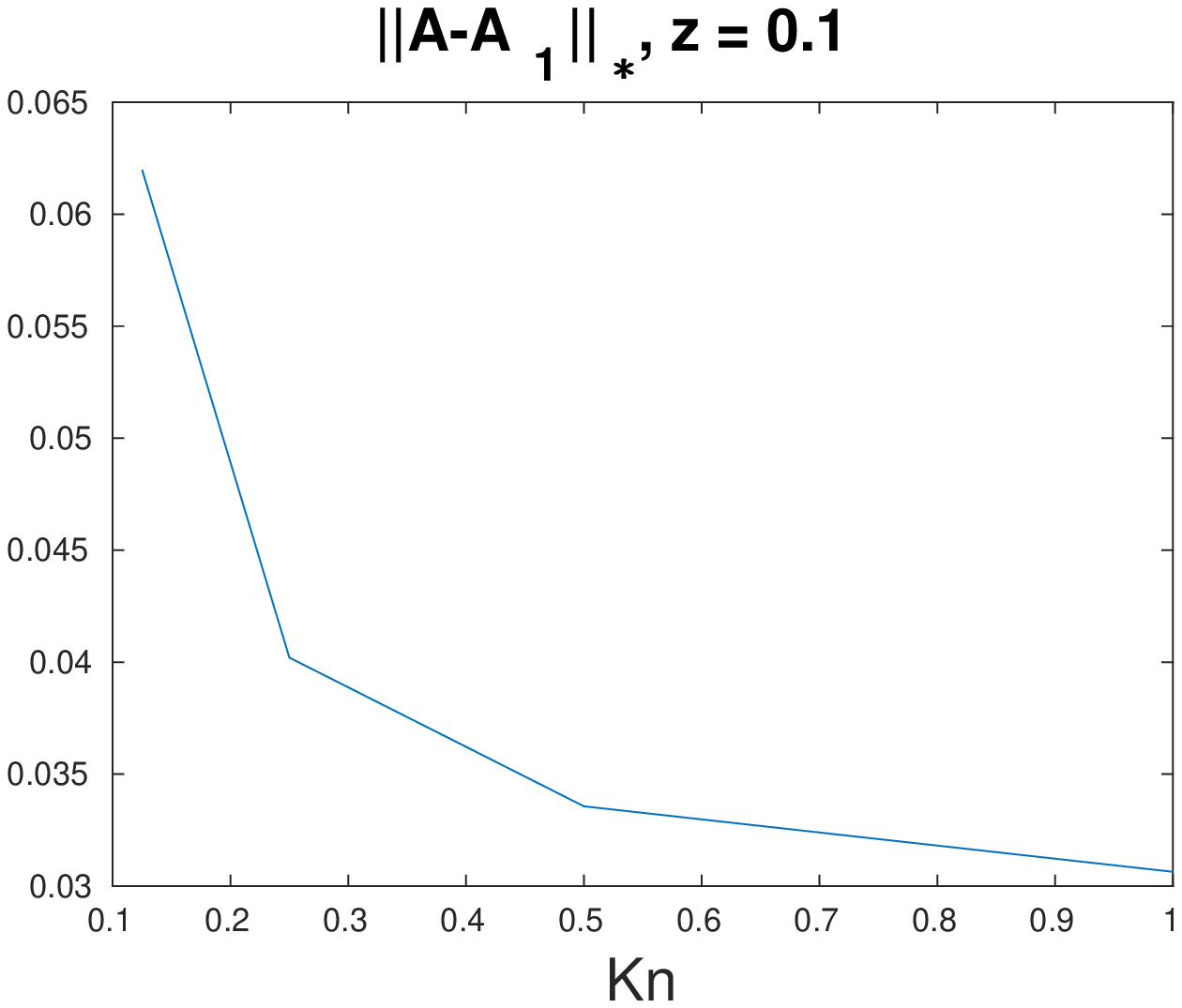}
	\caption{The plot on the left shows that $\ln\|\mathcal{A}_1\|_\ast$ linearly decays as $1/\Kn$ increases. The plot on the right shows that $\|\mathcal{A}-\mathcal{A}_1\|_\ast$ blows up as $\Kn$ converges to zero.  }\label{fig:A_A1}
\end{figure}

\noindent\textbf{Lipschitz continuity in $z$:}\\
In the second experiment we set $\Kn=1$ and study the dependence of $\|\mathcal{A} - \tilde{\mathcal{A}}\|_\ast$ on $z = \|\sigma_s-\tilde{\sigma}_s\|_{L^\infty}$. The numerical experiment suggests that the discrepancy between the two albedo operators increase linearly with respect to $z$, which agrees with our Lipschitz continuity result. 
\begin{figure}[htb]
	\includegraphics[height = 2.5in]{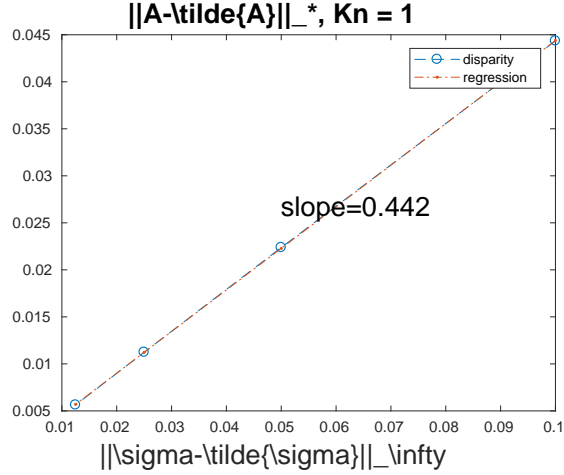}
	\caption{The plot shows that for fixed $\Kn=1$, bigger $\|\sigma_s-\tilde{\sigma}_s\|$ leads to bigger $\|\mathcal{A}-\tilde{\mathcal{A}}\|_\ast$ and they form a linear dependence.}\label{fig:L1_z}
\end{figure}

\noindent\textbf{Exponential blow-up in $\Kn$:}\\
In the third experiment, we fix $z=0.025$ and compare the difference between the two albedo operators $\mathcal{A}$ and $\tilde{\mathcal{A}}$ as a function of $\Kn$. It is expected that the difference between the two decays as $e^{-\frac{c}{\Kn}}$, according to Theorem~\ref{thm1.1}, which is also what we observe numerically. As seen in Figure~\ref{fig:l1}, $\ln\left(\|\mathcal{A} - \tilde{\mathcal{A}}\|_\ast\right)$ is a linear function of $1/\Kn$, with slope $-0.05$. This means:
\begin{equation*}
\ln{ \|\mathcal{A}-\tilde{\mathcal{A}}\|_\ast} \sim -\frac{0.05}{\Kn}\quad\Rightarrow\quad \|\mathcal{A}-\tilde{\mathcal{A}}\|_\ast \sim e^{-\frac{0.05}{\Kn}}\,.
\end{equation*}
\begin{figure}[htb]
	\includegraphics[height = 2.5in]{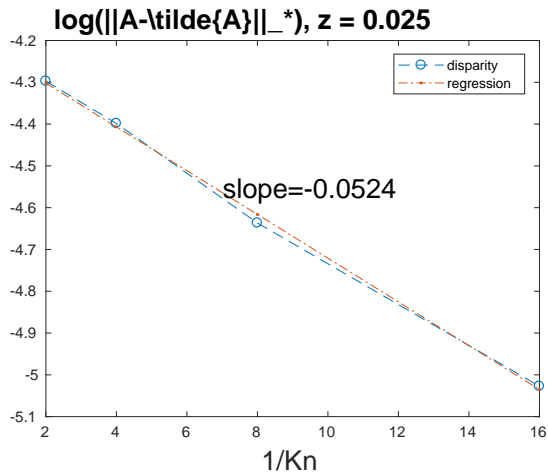}
	\caption{The plot shows that as $\Kn\to0$, $\frac{1}{\Kn}\to\infty$, and $\|\mathcal{A}-\tilde{\mathcal{A}}\|_\ast \sim e^{-\frac{0.05}{\Kn}}$.}\label{fig:l1}
\end{figure}


We emphasize before finishing the section that the numerical experiment can be done for only limited choices of $\Kn$, $\sigma_s$ and $z$, but the theory gives an upper bound for all possible combinations. It is also possible to design special media whose inverse stability is better than suggested by the theorem.

\vskip1cm
\noindent\textbf{Acknowledgements.}
The three authors would like to thank IMA for organizing the workshop ``Optical Imaging and Inverse Problems" February 13 - 17, 2017, during which the project was initiated. R.-Y. Lai is partially supported by the start-up grant from the University of Minnesota. Q. Li is partially supported by NSF grant DMS-1619778 and TRIPODS 1740707. G. Uhlmann is supported in part by NSF grant  DMS-1265958 and  a Si-Yuan Professorship at HKUS.

\bibliographystyle{abbrv}

\bibliography{transbib}

\begin{thebibliography}{10}

\bibitem{A1}
G.~Alessandrini.
\newblock Stable determination of conductivity by boundary measurements.
\newblock {\em Appl. Anal.}, 27:153--172, 1988.

\bibitem{Astability}
G.~Alessandrini.
\newblock Open issues of stability for the inverse conductivity problem.
\newblock {\em J. Inverse Ill-posed Problems}, 15:451--460, 2007.

\bibitem{Arridge1999}
S.~R. Arridge.
\newblock Optical tomography in medical imaging.
\newblock {\em Inverse Problems}, 15:R41--R93, 1999.

\bibitem{Bal_transport}
G.~Bal.
\newblock Inverse transport theory and applications.
\newblock {\em Inverse Problems}, 25:053001, 2009.

\bibitem{Bal14}
G.~Bal and A.~Jollivet.
\newblock Stability estimates in stationary inverse transport.
\newblock {\em Inverse problems and Imaging}, 2:427--454, 2008.

\bibitem{Bal10}
G.~Bal and A.~Jollivet.
\newblock Stability estimates for time-dependent inverse transport.
\newblock {\em SIAM J. Math. Anal.}, 42(2):679--700, 2010.

\bibitem{Bal18}
G.~Bal and A.~Jollivet.
\newblock Generalized stability estimates in inverse transport theory.
\newblock {\em Inverse problems and Imaging}, 12(1):59--90, 2018.

\bibitem{time_harmonic}
G.~Bal, A.~Jollivet, I.~Langmore, and F.~Monard.
\newblock Angular average of time-harmonic transport solutions.
\newblock {\em Communications in Partial Differential Equations},
  36(6):1044--1070, 2011.

\bibitem{BalMonard_time_harmonic}
G.~Bal and F.~Monard.
\newblock Inverse transport with isotropic time-harmonic sources.
\newblock {\em SIAM Journal on Mathematical Analysis}, 44(1):134--161, 2012.

\bibitem{BardosSantosSentis:84}
C.~Bardos, R.~Santos, and R.~Sentis.
\newblock Diffusion approximation and computation of the critical size.
\newblock {\em Trans. Amer. Math. Soc.}, 284(2):617--649, 1984.

\bibitem{BLP:79}
A.~Bensoussan, J.~Lions, and G.~Papanicolaou.
\newblock Boundary-layers and homogenization of transport processes.
\newblock {\em J. Publ. RIMS Kyoto Univ.}, 15:53--157, 1979.

\bibitem{CLW}
K.~Chen, Q.~Li, and L.~Wang.
\newblock Stability of stationary inverse transport equation in diffusion
  scaling.
\newblock {\em Inverse Problems}, 34(2), 2018.

\bibitem{CS1}
M.~Choulli and P.~Stefanov.
\newblock Scattering inverse pour l'\'equation du transport et relations entre
  les op\'erateurs de scattering et d'alb\'edo.
\newblock {\em C. R. Acad. Sci. Paris}, 320:947--952, 1995.

\bibitem{CS2}
M.~Choulli and P.~Stefanov.
\newblock Inverse scattering and inverse boundary value problems for the linear
  boltzmann equation.
\newblock {\em Comm. P.D.E.}, 21:763--785, 1996.

\bibitem{CS3}
M.~Choulli and P.~Stefanov.
\newblock Reconstruction of the coefficients of the stationary transport
  equation from boundary measurements.
\newblock {\em Inverse Problems}, 12:L19--L23, 1996.

\bibitem{CS98}
M.~Choulli and P.~Stefanov.
\newblock An inverse boundary value problem for the stationary transport
  equation.
\newblock {\em Osaka J. Math.}, 36:87--104, 1998.

\bibitem{Colak}
S.~B. Colak, D.~G. Papaioannou, W.~G. Hooft, M.~B. Van~der Mark, H.~Schomberg,
  J.~C.~J. Paasschens, J.~B.~M. Melissen, and N.~A. A.~J. Van~Asten.
\newblock Tomographic image reconstruction from optical projections in
  light-diffusing media.
\newblock {\em Appl. Opt.}, 36:180--213, 1997.

\bibitem{DDA}
H.~Dehghani, D.~T. Delpy, and S.~R. Arridge.
\newblock Photon migration in non-scattering tissue and the effects on image
  reconstruction.
\newblock {\em Phys. Med. Biol.}, 44:2897--2906, 1999.

\bibitem{Fantini}
S.~Fantini, M.~A. Franceschini, G.~Gaida, E.~Gratton, H.~Jess, W.~W. Mantulin,
  K.~T. Moesta, P.~Schlag, and M.~Kaschke.
\newblock Photon migration in non-scattering tissue and the effects on image
  reconstruction.
\newblock {\em Med. Phys.}, 23:149--157, 1996.

\bibitem{HA1996}
A.~H. Hielscher and R.~E. Alcouffe.
\newblock Non-diffusive photon migration in homogenous and heterogenous
  tissues.
\newblock {\em Proc. SPIE}, 2925:22--30, 1996.

\bibitem{I}
V.~Isakov.
\newblock Increasing stability for the {S}chr\"odinger potential from the
  dirichlet-to-neumann map.
\newblock {\em DCDS-S}, 4:631--640, 2011.

\bibitem{Ibook}
V.~Isakov.
\newblock {\em Inverse Problems for Partial Differential Equations}.
\newblock Springer-Verlag, New York, 2017.

\bibitem{ILW16}
V.~Isakov, R.-Y. Lai, and J.-N. Wang.
\newblock Increasing stability for the conductivity and attenuation
  coefficients.
\newblock {\em SIAM J. Math. Anal.}, 48(1):569--594, 2016.

\bibitem{INUW}
V.~Isakov, S.~Nagayasu, G.~Uhlmann, and J.-N. Wang.
\newblock Increasing stability of the inverse boundary value problem for the
  {S}chr\"odinger equation.
\newblock {\em Contemp. Math.}, 615:131--141, 2014.

\bibitem{iw14}
V.~Isakov and J.-N. Wang.
\newblock Increasing stability for determining the potential in the
  {S}chr\"odinger equation with attenuation from the dirichlet-to-neumann map.
\newblock {\em Inverse Problems and Imaging}, 8:1139--1150, 2014.

\bibitem{Somersalo2005}
J.~Kaipio and E.~Somersalo.
\newblock {\em Statistical and Computational Inverse Problems}.
\newblock Springer, 2005.

\bibitem{Ldiff}
R.-Y. Lai.
\newblock Increasing stability for the diffusion equation.
\newblock {\em Inverse Problems}, 30:075010, 2014.

\bibitem{LiLuSun2015JCP}
Q.~Li, J.~Lu, and W.~Sun.
\newblock Diffusion approximations of linear transport equations: Asymptotics
  and numerics.
\newblock {\em J. Comp. Phys}, 292:141--167, 2015.

\bibitem{LLS_geometry}
Q.~Li, J.~Lu, and W.~Sun.
\newblock Validity and regularization of classical half-space equations.
\newblock {\em Journal of Statistical Physics}, 166(2):398--433, 2017.

\bibitem{Liang}
L.~Liang.
\newblock Increasing stability for the inverse problem of the {S}chr\"odinger
  equation with the partial cauchy data.
\newblock {\em Inverse Problems and Imaging}, 9:469--478, 2015.

\bibitem{LN1983}
A.~Louis and F.~Natterer.
\newblock Mathematical problems of computerized tomography.
\newblock {\em Proceedings of the IEEE}, 71(3):379--389, 1983.

\bibitem{Mandache}
N.~Mandache.
\newblock Exponential instability in an inverse problem for the {S}chr\"odinger
  equation.
\newblock {\em Inverse Problems}, 17:1435--1444, 2001.

\bibitem{McCor}
N.~J. McCormick.
\newblock Inverse radiative transfer problems: A review.
\newblock {\em Nuclear Sci. Engrg.}, 112:185--198, 1992.

\bibitem{NUW}
S.~Nagayasu, G.~Uhlmann, and J.-N. Wang.
\newblock Increasing stability in an inverse problem for the acoustic equation.
\newblock {\em Inverse Problems}, 29:025012, 2013.

\bibitem{Stefanov_2003}
P.~Stefanov.
\newblock {\em Inverse problems in transport theory}, volume~47.
\newblock Inside Out: Inverse Problems; MSRI Publications, edited by G.
  Uhlmann, 2003.

\bibitem{SU2d}
P.~Stefanov and G.~Uhlmann.
\newblock Optical tomography in two dimensions.
\newblock {\em Methods Appl. Anal.}, 10:1--9, 2003.

\bibitem{uhlmann2009electrical}
G.~Uhlmann.
\newblock Electrical impedance tomography and {C}alder{\'o}n's problem.
\newblock {\em Inverse problems}, 25(12):123011, 2009.

\bibitem{Wang1999}
J.-N. Wang.
\newblock Stability estimates of an inverse problem for the stationary
  transport equation.
\newblock {\em Ann. Inst. H. Poincar\'e Phys. Th\'eor.}, 70(5):473--495, 1999.

\bibitem{WG2014}
L.~Wu and Y.~Guo.
\newblock Geometric correction for diffusive expansion of steady neutron
  transport equation.
\newblock {\em Communications in Mathematical Physics}, 336(3):1473--1553,
  2015.

\end{thebibliography}

\end{document}